\documentclass[11pt,a4paper]{amsart}
\usepackage[english]{babel}
\usepackage{lmodern}
\usepackage{newlfont}
\usepackage{booktabs}
\usepackage{color}
\usepackage[T1]{fontenc}
\usepackage[utf8]{inputenc}
\usepackage{indentfirst}
\usepackage{amsmath}
\usepackage{amsmath,amssymb}
\usepackage{amsthm}
\usepackage{geometry}
\usepackage{mathtools}
\usepackage{listings}
\usepackage{amsfonts}
\usepackage{braket}
\usepackage{emptypage}
\usepackage{newlfont}
\usepackage{amssymb}
\usepackage{graphicx}
\usepackage[italian]{varioref}
\usepackage{accents}
\usepackage{comment}
\usepackage{stmaryrd}
\usepackage{pgfplots}

\makeatletter
\labelformat{equation}{\tagform@{#1}}
\makeatother
\usepackage{hyperref}
\usepackage{relsize}
\usepackage{faktor}
\usepackage{enumerate}
\usepackage{fancyhdr}

\usepackage[colorinlistoftodos]{todonotes}

\DeclarePairedDelimiter{\abs}{\lvert}{\rvert}

\DeclareMathOperator{\divv}{div}

\renewcommand{\epsilon}{\varepsilon}

\newcommand{\eu}{\mathbf{u}}

\newcommand{\numberset}{\mathbb}
\newcommand{\eps}{\varepsilon}

\newcommand{\N}{\numberset{N}}

\newcommand{\R}{\numberset{R}}

\newcommand{\norm}[1]{\|#1\|}
\newcommand{\wto}{\rightharpoonup}

\theoremstyle{definition}

\theoremstyle{definition}                                                                         
\newtheorem{definizione}{Definizione}[section]
\theoremstyle{definition}                                                                         

\theoremstyle{plain}        
\newtheorem{rmks}[definizione]{Remarks}
\theoremstyle{plain}                                                                            
\newtheorem{thm}[definizione]{Theorem}
\theoremstyle{plain}    
\newtheorem{prop}[definizione]{Proposition}
\theoremstyle{plain}     
\newtheorem{lemma}[definizione]{Lemma}
\theoremstyle{plain}
\newtheorem{cor}[definizione]{Corollary}
\theoremstyle{definition}

\numberwithin{equation}{section}

\begin{document}

\title[Monotone heteroclinic solutions to semilinear PDEs]{Monotone heteroclinic solutions to semilinear PDEs in cylinders and applications}

 \author{Fabio De Regibus}
\address{ \vspace{-0.4cm}
	\newline 
	\textbf{{\small Fabio De Regibus}} 
	\vspace{0.15cm}
	\newline \indent  
	Departamento de An\'alisis Matem\'atico, Universidad de Granada, 18071 Granada, Spain}
\email{fabioderegibus@ugr.es}

\author{David Ruiz}
\address{ \vspace{-0.4cm}
	\newline 
	\textbf{{\small David Ruiz}} 
	\vspace{0.15cm}
	\newline \indent  
	IMAG, Departamento de An\'alisis Matem\'atico, Universidad de Granada, 18071 Granada, Spain}
\email{daruiz@ugr.es}



\thanks{ D.R. has been supported by:
	the Grant PID2021-122122NB-I00 of the MICIN/AEI, the \emph{IMAG-Maria de Maeztu} Excellence Grant CEX2020-001105-M funded by MICIN/AEI, and the Research Group FQM-116 funded by J. Andaluc\'\i a.
F.D.R. has been supported by:
	the Juan de la Cierva fellowship Grant JDC2022-048890-I funded by MICIU/AEI/10.13039/501100011033 and by the “European Union NextGeneration EU/PRTR”.}

\begin{abstract}
In this paper we show the existence of strictly monotone heteroclinic type solutions of semilinear elliptic equations in cylinders. The motivation of this construction is twofold: first,  it implies the existence of an entire bounded solution of a semilinear equation without critical points which is not one-dimensional. Second, this gives an example of a bounded stationary solution for the 2D Euler equations without stagnation points which is not a shear flow, completing previous results of Hamel and Nadirashvili. The proof uses a minimization technique together with a truncation argument, and a limit procedure.
\end{abstract}

\maketitle

\section{Introduction and main results}

Let $N\ge2$ and consider a smooth and bounded domain $\omega\subseteq\R^{N-1}$. We denote by $\Omega$ the unbounded $N$-dimensional cylinder of section $\omega$, i.e.,
\[
\Omega=\omega\times\R\subseteq\R^N.
\]
We are interested in solutions of the following semilinear elliptic problem:
\begin{equation}
\label{PB}
\begin{cases}
-\Delta u=f(u)&\text{in }\Omega,\\
u=0&\text{on }\partial\Omega,
\end{cases}
\end{equation}
for some nonlinear term $f \in  C^1(\R)$. Here, and in the sequel, we write $x=(x',x_N)$ for a generic point in $\R^N$, where $x'=(x_1,\dots,x_{N-1})\in\R^{N-1}$, and we denote $e_N =(0, \dots , 0, 1)$.

 In particular we focus on heteroclinic type solutions of problem~\ref{PB}: given two different solutions $\varphi_1,\varphi_2$ we are looking for another solution $u$ linking $\varphi_1$ and $\varphi_2$ in the sense that
\begin{equation} \label{limits}
\lim_{\lambda\to-\infty}u(\cdot + \lambda e_N)=\varphi_1(\cdot),\quad\text{and}\quad \lim_{ \lambda \to+\infty}u(\cdot + \lambda e_N)=\varphi_2(\cdot).
\end{equation}

In the literature there are many works about heteroclinic solutions for elliptic PDEs in different contexts. In particular, we would like to mention the papers~\cite{Ra94,Ra02} by P. H. Rabinowitz, where the existence of heteroclinic solutions in the strip is shown both for the case of Neumann and Dirichlet boundary conditions, in a periodic setting: that is, the limit functions $\varphi_1$ and $\varphi_2$ are periodic in the $x_N$ variable. For further results in cylinders we refer also to~\cite{A19, BMR16, Ki82,Ra04}. A lot of work has been addressed also in the case of heteroclinic solutions for a non autonomous Allen-Cahn type problem in the whole space $\R^N$, see for instance~\cite{AGM16,AJM00,AJM02,AM05,AM07,AM13,RS03,RS04} and the references therein.

In this paper, instead, we are looking for (strictly) monotone solutions of problem~\ref{PB},~\ref{limits}, where $\varphi_1=0$ and $\varphi_2= \varphi>0$ are independent of $x_N$, that is, they are solutions of the $N-1$ dimensional problem on $\omega$:
\begin{equation}
\label{n-1:PB}
\begin{cases}
-\Delta_{x'}\phi=f(\phi)&\text{in }\omega,\\
\phi=0&\text{on }\partial\omega.
\end{cases}
\end{equation}
Here we denote by $\Delta_{x'}$ the Laplacian with respect to the first $N-1$ coordinates. Clearly problem~\ref{n-1:PB} is the Euler-Lagrange equation of the action functional $I\colon H_0^1(\omega) \to \R$,
\begin{equation} \label{defI}
I(u)=\int_\omega \left( \frac12\abs{\nabla_{x'} u}^2-F(u) \right) \,dx',  \quad \mbox{ where } F(u)=\int_0^uf(s)\,ds.
\end{equation}

In our main result we make the following assumptions:
\begin{equation}
\label{H1} \tag{H1} I(0)= \inf \left \{I(u):  u \in H_0^1(\omega) \right \}=0 \mbox{ and } 0 \mbox{ is an isolated minimizer of }I;
\end{equation}
\begin{equation}
\label{H2} \tag{H2}
\mbox{ there exists } \psi \in H_0^1(\omega), \ \psi >0  \mbox{ such that }  I(\psi)=0.
\end{equation}

In other words, we are assuming that $0$ and $\psi$ are global minimizers of $I$, and that $0$ is isolated. In particular, $0$ and $\psi$ are solutions of~\ref{n-1:PB}, which implies $f(0)=0$. Moreover, it can be proved (see Lemma~\ref{minimal}) that under these assumptions there exists a solution to~\ref{n-1:PB}, denoted by  $\varphi$, which is minimal among the set of global minimizers. In other words, $\varphi \in H_0^1(\omega)$ is a positive function such that $I(\varphi) =0$, and if $\psi\in H_0^1(\omega)$, $\psi>0$, satisfies $I(\psi)=0$, then $\varphi \leq \psi$.

The following is our main result:

\begin{thm}
\label{tmain}
Assume that $f \in  C^1(\R)$ is such that~\ref{H1},~\ref{H2} are satisfied.
Then, there exists $u\in C^{2, \alpha}(\overline\Omega)$ solution of~\ref{PB} such that
\[
0<u<\varphi\quad\text{in }\Omega.
\]
Moreover, u is heteroclinic from $0$ to $\varphi$, i.e.
\[
\lim_{x_N\to-\infty}u(\cdot,x_N)=0,\quad\text{and}\quad \lim_{x_N\to+\infty}u(\cdot,x_N)=\varphi(\cdot),
\]
uniformly in $x'\in\omega$. Finally, $u$ is strictly increasing in $x_N$, that is,
\[
\partial_{x_N} u>0\quad\text{in }\Omega.
\]
\end{thm}

This result is strongly related to~\cite[Section 4]{Ra02}, where the $x_N$-periodic case is treated. In our case we are interested in deriving also the monotonicity property, which will be essential in the applications of Theorem~ \ref{tmain} that will be presented below.

A comment on our set of hypotheses~\ref{H1} and~\ref{H2} is in order. First, given any bounded smooth domain $\omega \subseteq \R^{N-1}$, we can build $f$ so that~\ref{H1} and~\ref{H2} are satisfied, see Proposition~\ref{example}. Moreover we shall see that, to some extent, assumptions~\ref{H1} and~\ref{H2} are necessary for the existence of the heteroclinic solution given in Theorem~\ref{tmain}, see Proposition~\ref{Hi-nec}. \\

The proof of Theorem~\ref{tmain} follows from a truncated minimization procedure. First, we consider the family of bounded domains:
\[
\Omega_n=\omega\times(-n,n).
\]

For each one of those domains we will consider the problem of finding a solution to the following:
\begin{equation}\label{PBn}
	\begin{cases}
		-\Delta u_n=f(u_n)&\text{in }\Omega_n,\\
		u_n=0&\text{on }\partial \Omega_n\setminus(\omega\times\{n\}),\\
		u_n=\varphi&\text{on }\omega\times\{n\}.
	\end{cases}
\end{equation}

These solutions will be searched as minimizers of the corresponding energy functional. Then, we make $n \to +\infty$ and show that, up to a suitable translation, $u_n$ converges locally to our desired solution. In our proofs the use of a Hamiltonian identity (see for instance~\cite{Gui}) will be a crucial tool.

Our main motivation for this paper comes from two related questions, that we discuss below.

\subsection{Entire solutions of semilinear equations}
Let us consider a bounded entire solution of the problem:
\begin{equation}
\label{semi}
-\Delta u=f(u) \quad\text{in }\R^2,
\end{equation}
for some function $f \in  C^1(\R)$. It is well known that if the following monotonicity condition holds:
\begin{equation}
\label{monotone}
 \partial_{x_N}u(x)>0 \quad\mbox{ for all  } x \in \R^2, 
 \end{equation} then  $u$ is $1$-dimensional, that is, its level sets are hyperplanes (see~\cite{GG98}). In the case of the Allen-Cahn nonlinearity $f(u)=u-u^3$ this is the statement of the well known De Giorgi conjecture. The De Giorgi conjecture has been proved to hold also in dimension 3 (\cite{AC00}); instead, in dimension $N \geq 9$ there are solutions of the problem $-\Delta u = u-u^3$ satisfying~\ref{monotone} which are not $1$-dimensional, see~\cite{DePKW}. For $4 \leq N \leq 8$ the problem is open: it is known that monotone solutions to the Allen-Cahn equation are $1$-dimensional under certain additional conditions, see~\cite{Sa09}.

Back in dimension 2, a similar result can be proved for stable solutions: that is, if $u$ is a stable solution, then it is 1-dimensional\footnote{By the way, the extension of this result to higher dimensions, even to dimension 3, is a very important open problem.}. Here stability means that $Q(\phi) \geq 0$ for all $\phi \in C_0^{\infty}(\R^2)$, where $Q$ is the quadratic form associated to the linearized operator: 
\[
 Q(\phi) = \int_{\R^2}\left( |\nabla \phi|^2 - f'(u) \phi^2\right)\,dx.
 \]
It is well known that the monotonicity property~\ref{monotone} implies stability: for this and other questions in this framework we refer to the monograph~\cite{Dupaigne}.

Hence, one could think of relaxing hypothesis~\ref{monotone} in a different direction, at least in dimension 2. A natural possibility could be replacing the monotonicity condition by:
\begin{equation}
\label{nocrit} 
\nabla u(x) \neq 0 \quad\mbox{ for all  }  x \in \R^2.
\end{equation}

In other words, is it true that bounded solutions to~\ref{semi} without critical points are 1-dimensional? 

The following observation is contained in~\cite{Far03}: under condition~\ref{nocrit}, we can write $\nabla u(x) = \rho(x) e^{i \theta(x)}$, by using complex notation. By regularity estimates $\rho$ is uniformly bounded, and it turns out that 
\begin{equation} \label{BCN} \divv (\rho^2 \nabla \theta)=0. \end{equation} In~\cite{BCN97} Berestycki, Caffarelli and Nirenberg gave a Liouville type result for equation~\ref{BCN}. In our case this result implies that if $\theta$ is bounded then $\theta $ is necessarily constant, and this implies that $u$ is 1-dimensional. Observe that~\ref{monotone} implies that $\theta(x) \in (0, \pi)$, and this argument was used in \cite{Far03} to give an alternative proof of the De Giorgi conjecture in dimension 2. 

In  general, one can obtain the same conclusion under some control on the growth of $\theta$ (see \cite{Salva} for the sharp growth condition that can be admitted).  
In~\cite{HN19} the authors assume $0<c < \rho(x)$ for all $x \in \R^2$: with this assumption the operator in~\ref{BCN} is uniformly elliptic, and by using Harnack estimates they can estimate the growth of $\theta$ and conclude that $\theta$ is constant. But this argument does not work if we only assume~\ref{nocrit}. 

As a consequence of Theorem~\ref{tmain}, we shall show an example of a solution to~\ref{semi} without critical points which is not 1-dimensional. Being more specific, we take $\omega=(0,1)$, and extend the solution given in Theorem~\ref{tmain} to the whole plane by odd reflection. In doing so we obtain an entire solution with the required properties (see Theorem~\ref{nocrit} for details). 

We point out that if $\Omega = (0,1) \times \R$, the solution given in Theorem~\ref{tmain} is an example of a solution of~\ref{PB} with $\{\theta(x):  x \in \Omega\} = (0, \pi)$, as that of~\cite[Theorem 1.3]{GXX24}.

\subsection{Stationary solutions to the 2D Euler equations}
In the recent years a lot of work has been devoted to the study of rigidity result for solutions of the stationary Euler equation
\begin{equation}
\label{Eu}
\begin{cases}
\eu\cdot\nabla \eu=-\nabla p&\text{in }\Omega,\\
\divv \eu=0&\text{in }\Omega,
\end{cases}
\end{equation}
where $\Omega$ is an open subset of $\R^2$. Here $\eu=(\eu^1,\eu^2)$ is the velocity vector field of an inviscid and incompressible fluid and $p$ is the pressure.

Several results regarding rigidity and non rigidity for the stationary Euler equations can be found in~\cite{EFR24,GSPS21,GSPSJY21, GXX24, HK23,HN21,LWX23,LLSX23,Ru22,WZ23}. In~\cite{HN19}, Hamel and Nadirashvili consider bounded solutions of~\ref{Eu} without stagnation points in $\R^2$   nor "at infinity", that is $\inf \ \abs \eu>0$. Under these assumptions, the authors prove that $\eu$ is a shear flow, that means, $\eu(x)={U}(x\cdot e^\perp)e$, for some $e=(e_1,e_2)\in\mathbb S^1$ and $e^\perp=(-e_2,e_1)$. Moreover, they show, by providing explicit counterexamples, that if $\eu$ has some stagnation point, then it is not necessarily a shear flow. Similar results have been obtained by the same authors also in the half-plane $\R^+\times\R$ and in the strip $(0,1) \times \R$ (\cite{HN17}) under non-slip boundary conditions $\eu \cdot \eta =0$, where $\eta$ is the outer normal vector on $\partial \Omega$.

Hence, a natural question is whether such results hold true assuming only the absence of stagnation points, that is, $\eu (x) \neq 0$ for all $x \in \overline{\Omega}$. In other words we allow the presence of stagnation points "at infinity", that is, $\inf_{\R^2}\abs{\eu}=0$.
As a consequence of our main result, we can build bounded solutions in a strip, a half-space or the whole space $\R^2$, without stagnation points, which are not shear flows (Corollary~\ref{cor}). The idea is to look for streamfunctions which are solutions to a semilinear elliptic equation as commented above.

\subsection{Organization of the paper}

The rest of the paper is organized as follows: in the next section we establish some preliminaries and prove Theorem~\ref{tmain}. Section~\ref{sec3} is devoted to a discussion on our assumptions~\ref{H1} and~\ref{H2}. In particular, we show that we can build nonlinear functions $f(u)$ such that those conditions are satisfied. We also show that~\ref{H1} and~\ref{H2} imply the existence of a minimal solution, and finally we point out that, to some extent, these assumptions are necessary for Theorem~\ref{tmain}. In Section~\ref{sec4} we present the applications of Theorem~\ref{tmain} to the questions mentioned above, namely, entire solutions of semilinear equations without critical points and steady solutions of the 2D Euler equations without stagnation points.

\section{Proof of Theorem~\ref{tmain}}
\label{sec2}

First of all, let us state a Hamiltonian-type identity for solutions on a cylinder, in the spirit of~\cite{Gui}. We give the proof below for the sake of completeness.

\begin{prop} \label{H} Let $A \subseteq \R$ an interval, and $\Omega_A = \omega \times A$. For any $t \in A$ we denote by $\omega_t= \{(x', t): \ x' \in \omega\}$. Let $u$ be a $ C^2$ solution of the problem:
\begin{equation*}
	\begin{cases}
		-\Delta u=f(u)&\text{in }\Omega_A,\\
		u=0&\text{on }\partial\omega \times A.
	\end{cases}
\end{equation*}
Then, the Hamiltonian:
\[
H= \int_{\omega_t} \left(\frac 12 \left( |\nabla_{x'} u|^2 - (\partial_{x_N}u)^2 \right) - F(u) \right)\,dx',
\]
is independent of $t$.
\end{prop}

\begin{proof}
It suffices to write $H$ as:
\[
 H(t)= \int_{\omega}	\left(\frac 12 \left( |\nabla_{x'} u(x',t)|^2 - (\partial_{x_N}u(x',t))^2 \right) - F(u(x',t)) \right)\, dx'.
 \]
Then we compute the derivative with respect to $t$ and integrate by parts to conclude:
\begin{align*}
 \frac{d}{dt} H(t)&= \int_{\omega} \left(	\nabla_{x'} u \cdot 	\nabla_{x'} (\partial_{x_N}u) - \partial_{x_N}u \,  \partial^2_{x_N x_N}u - f(u) \partial_{x_N}u \right) \, dx'\\
	& = \int_{\omega} \left(	\nabla_{x'} u \cdot 	\nabla_{x'} (\partial_{x_N}u) + \partial_{x_N}u \,  \Delta_{x'}u  \right) \, dx'=0.
\end{align*}

\end{proof}

We now point out that under assumptions~\ref{H1} and~\ref{H2}, there exists a minimal positive minimizer of $I$, which will be denoted by $\varphi$. This must be rather well known, but we have not found an explicit reference, so we include its proof in Section 3 (see Lemma~\ref{minimal}).

We are now in conditions to address the proof of Theorem~\ref{tmain}. The general strategy is to build the solution $u$ of the problem~\ref{PB} by approximation on the domain $\Omega$. At least formally, the energy functional associated to problem~\ref{PB} is
\[
J(u)=\int_\Omega\mathcal L(u)\,dx,
\]
where
\[
\mathcal L(u)=\frac12\abs{\nabla u}^2-F(u), \quad F(u)=\int_0^uf(s)\,ds.
\]
Let us recall that, for any  $n\in\N$, we defined the approximating domains:
\[
\Omega_n=\omega\times(-n,n).
\]
Hence we set
\[
\mathcal{H}_n=\Set{u\in H^1(\Omega_n)|u=0\text{ on }\partial\Omega_n\setminus\{x_N=n\},\quad u=\varphi\text{ on }\{x_N=n\}},
\]
where $\varphi$ is the minimal solution given by Lemma~\ref{minimal}.

By taking a suitable function $\xi\colon[-n,n] \to \R$ with $\xi(-n)=0$ and $\xi(n)=1$, we have that $\phi(x', x_N) = \xi(x_N) \varphi(x') \in \mathcal{H}_n$, so that $\mathcal{H}_n$ is not empty. Then we consider the restriction of the functional $J$ to the set $\mathcal{H}_n$,
\[
J_n\colon \mathcal{H}_n \to \R, \quad J_n(u)=\int_{\Omega_n}\mathcal L(u)\,dx.
\]

The first step in the proof of Theorem~\ref{tmain} is the following:

\begin{prop} For any $n \in \N$, we denote $ c_n = \inf \ J_n$.	
Then, the following assertions hold true:
\begin{enumerate}
	\item $c_n >0$ and $c_n$ is decreasing in $n$.
	\item There exists $u_n \in \mathcal{H}_n$ such that $J_n(u_n)=c_n$. In particular, $u_n$ is a solution of~\ref{PBn}.
	\item $0<u_n(x', x_N) <\varphi(x')$ for all $(x', x_N) \in \Omega_n$. 
	\item $H_n < 0$ where $H_n$ denotes the value of the Hamiltonian defined in Proposition~\ref{H}.
	\item $u_n \in  C^{2, \alpha}(\overline{\Omega_n})$, and $\partial_{x_N} u_n(x) >0$ for all $x \in \Omega_n$.
\end{enumerate}	

\end{prop}

\begin{proof}
	
We first show that $c_n \geq 0$. For any $u \in H^1(\Omega_n)$ with $u=0$ on $\partial\omega\times(-n,n)$ - and then in particular for all $u\in \mathcal{H}_n$ - there holds:
\begin{equation}
\label{c_n>0}
\begin{split}
J_n(u)= \int_{\Omega_n}\mathcal L(u(x))\,dx& = \int_{-n}^n \int_{\omega} \mathcal L(u(x',x_N)) \ dx' \ dx_N  \\& \geq  \int_{-n}^n I(u(\cdot,x_N)) \ dx_N \geq 0,
\end{split}
\end{equation}
where $I$ is defined in~\ref{defI}. 
Let $n, m \in \N$ with $n<m$, and take $u \in \mathcal{H}_n$. Then we can extend:
\[
 \tilde{u}(x', x_N) = \begin{cases}u(x', x_N)  & \mbox{ if } x_N \in (-n,n),  \\ 0  & \mbox{ if } x_N < -n,  \\ \varphi(x)  & \mbox{ if } x_N > n.   \end{cases} 
 \]

Clearly $\tilde{u} \in \mathcal{H}_m$ and $ J_m(\tilde u) =  J_n(u)$. In this way, we can embed $\mathcal{H}_n \subset \mathcal{H}_m$ and then $c_n \geq c_m$. 

In order to prove (2), consider a minimizing sequence $(v_k)_{k\in\N}\subseteq \mathcal{H}_n$ such that $J_n(v_k)\to c_n$, as $k\to+\infty$. Then it is not hard to see that
\[
w_k(x)=
\begin{cases}
	0,&\text{if }v_k(x)\le0,\\
	v_k(x),&\text{if } 0<v_k(x)<\varphi(x),\\
	\varphi(x),&\text{if } v_k(x)\ge\varphi(x),
\end{cases}
\]
is still a minimizing sequence. Indeed taking into account that 
$(v_k)^-=\min\{v_k,0\}$, $\max\{\varphi,v_k\}\in H^1(\Omega_n)$ and they both vanish on $\partial\omega\times(-n,n)$,~\ref{c_n>0} implies
\begin{gather*}
\int_{\{v_k\le0\}}\mathcal L(v_k)\,dx=J_n((v_k)^-)\ge0,\\
\int_{\{v_k<\varphi\}}\mathcal L(\varphi)\,dx+\int_{\{v_k\ge\varphi\}}\mathcal L(v_k)\,dx=J_n(\max\{\varphi,v_k\})\ge0,
\end{gather*}
and then
\begin{align*}
	J_n(w_k)	&=\int_{\{0<v_k<\varphi\}}\mathcal L(v_k)\,dx+\int_{\{v_k\ge\varphi\}}\mathcal L(\varphi)\,dx\\
	&\le\int_{\{v_k\le0\}}\mathcal L(v_k)\,dx+\int_{\{0<v_k<\varphi\}}\mathcal L(v_k)\,dx+\int_{\{v_k\ge\varphi\}}\mathcal L(\varphi)\,dx\\
	&\quad+\int_{\{v_k<\varphi\}}\mathcal L(\varphi)\,dx+\int_{\{v_k\ge\varphi\}}\mathcal L(v_k)\,dx\\
	&= J_n(v_k)+\int_{-n}^nI(\varphi)\,dx_N=J_n(v_k).
\end{align*}

In particular from the uniform $L^\infty$ boundedness of $w_k$, we get
\[
\left|\int_{\Omega_n} F(w_k)\,dx\right|\le C_n,
\]
for some positive constant $C_n$ independent of $k$. This estimate, together with the fact that $w_k$ is a minimizing sequence and $0\le w_k\le\varphi$ in $\Omega_n$, implies that $\norm{w_k}_{H^1(\Omega_n)}$ is uniformly bounded in $k$. Up to a subsequence we can assume that $w_k \rightharpoonup w$ in $H^1$ sense, and moreover, $w_k \to w$ pointwise. By the weak lower semicontinuity of the norm in $H^1$ and the dominated convergence theorem we obtain that $w$ is a minimizer, and (2) is proved. In what follows we will denote by $u_n$ the minimizer of $J_n$. As a byproduct we get $c_n >0$, completing the proof of (1).

Of course from the fact that $0\le w_k\le\varphi$ in $\Omega_n$ and the pointwise convergence of $w_k$ to $u_n$, one has $0\le u_n\le\varphi$ in $\Omega_n$.	The strict inequality stated in (3) will be a consequence of (5).
	
In order to prove (4) we just compute the Hamiltonian at $t=-n$:
\[
H_n = - \frac 1 2 \int_{\omega} (\partial_{x_N} u_n(x', -n))^2 \, dx' \leq 0.
\]

Moreover, if $H_n$ is equal to $0$, have that $\partial_{x_N} u_n(x', -n) =0$ for all $x' \in \omega$. By unique continuation this implies that $u_n \equiv 0$, which is impossible. 
	
We now prove (5). Standard regularity estimates imply that $u_n \in  C^{0, \alpha}(\overline{\Omega}_n) \cap  C^{2, \alpha} (\overline{\Omega}_n \setminus \mathcal N_\eps)$, where:
\[
 \mathcal N_{\eps} = \{(x', x_N) \in \overline{\Omega}_n: \ \mathrm{dist}(x', \partial \omega  ) + \min \{ |x_N-n|, |x_N+n| \} < \eps\}.
 \]	
We now show that the $ C^{2, \alpha}$ regularity extends to the whole domain. Indeed, define: $\zeta(x', x_N) = u_n(x', x_N)- \varphi(x')$. It is clear that
\[
- \Delta \zeta(x)= f(u_n(x)) - f(\varphi(x)) = h(x),
\]
 where $h \in  C^{0, \alpha}(\Omega_n)$, and moreover $\zeta(x', n)=0$, $h(x', n)=0$ for all $x' \in \omega$. We extend $\zeta$ and $h$ by reflection:
\[
\tilde{\zeta}(x', x_N)  =\begin{cases} \zeta(x',x_N) & \mbox{ if } x_N \leq n, \\ -\zeta(x', 2n-x_N) & \mbox{ if } n < x_N \leq 2 n, \end{cases} 
\]
\[
\tilde{h}(x', x_N)  = \begin{cases} h(x',x_N) & \mbox{ if } x_N \leq n, \\ -h(x', 2n-x_N) & \mbox{ if } n < x_N \leq 2 n. \end{cases} 
\]

Clearly $\tilde{h} \in  C^{0,\alpha} (\omega \times (-n, 2n))$ and $\tilde{\zeta}$ is a (weak, a priori) solution of the problem:
\[
 -\Delta \tilde\zeta(x)= \tilde h(x), \quad x \in \omega \times (-n,2n).
 \]

By local regularity up to the boundary, we conclude that $\tilde{\zeta}\in  C^{2,\alpha}$ in a neighborhood of $\omega \times \{n\}$, which implies regularity of $u_n$. In an analogous way we can argue around $\omega \times \{-n\}$, and we conclude.

The function $\partial_{x_N}u_n$ is a weak solution of the linearized problem
\[
\begin{cases}
	-\Delta (\partial_{x_N}u_n)=f'(u_n)\partial_{x_N}u_n&\text{in }\Omega_n,\\
	\partial_{x_N}u_n\ge0&\text{on }\partial\Omega_n,
\end{cases}
\]
where the boundary condition is satisfied since $u_n\equiv0$ on $\partial\omega\times(-n,n)$ and the fact that $0\le u_n\le\varphi$ in $\Omega_N$ implies $\partial_{x_N}u_n\ge0$ on $\omega\times\{-n,n\}$.

Observe now that, since $u_n$ is a global minimizer of $J_n$, 
\[
J_n''(u_n)(\phi, \phi) = \int_{\Omega_n}\left( |\nabla \phi |^2 - f'(u_n) \phi^2 \right)\,dx \geq 0, \quad\text{for all } \phi \in H_0^1(\Omega_n).
\]

In particular, $\lambda_1 \geq 0$, where $\lambda_1$ denotes the first eigenvalue of the operator $-\Delta - f'(u_n) $ under homogeneous Dirichlet boundary conditions. As a consequence this operator satisfies the maximum principle (see for instance~\cite{BNV}). Hence $\partial_{x_N}u_n \geq 0$ in $\Omega_n$, and equality holds only if $\partial_{x_N}u_n(x) \equiv 0$. But this is impossible since $0=u_n(x', -n) < u_n(x', n) = \varphi(x')$, and the strict inequality holds. Observe that this implies also the strict inequality in (3).

\end{proof}

It is now our intention to obtain a solution to~\ref{PB} as a limit of the functions $u_n$. This will be made by means of the Ascoli-Arzelà Theorem. However, we could obtain the trivial solutions $0$ or $\varphi$ in the limit: in order to avoid that, we make a convenient translation along the $x_N$ axis.

To this end, let us fix $z'\in\omega$ such that $\varphi(z')=\norm{\varphi}_{L^\infty(\omega)}$. Since, $u_n=0$ on $\omega\times\{-n\}$, $u_n=\varphi$ on $\omega\times\{n\}$ and $u_n$ is strictly monotone, there exists a unique $z_n\in(-n,n)$ such that
\[
u_n(z',z_n)=\frac12\varphi(z')=\frac12\norm{\varphi}_{L^\infty(\omega)}.
\]

The following lemma is a key step in passing to a limit:

\begin{lemma}
\label{lemmino}
We have that $n-z_n \to +\infty$ and $n+z_n \to +\infty$.
\end{lemma}

\begin{proof}

We reason by contradiction, and we assume that $n+z_n $ is not a diverging sequence. Passing to a subsequence and using  Ascoli-Arzelà Theorem we can assume that  $n + z_n \to z_0$ and the sequence $v_n(x',x_N) = u_n(x', x_N -n)$ converges locally in $ C^{2,\alpha}$ sense to a limit function $v$ which solves the problem:
\[
\begin{cases}
		-\Delta v=f(u)&\text{in }\omega \times \R^+ ,\\
		v=0&\text{on }\partial\left(\omega \times \R^+\right)	.
	\end{cases}
\]
Moreover, $v(z',z_0) = \frac12\norm{\varphi}_{L^\infty(\omega)}$, which implies in particular that $v \not\equiv 0$.

Furthermore, we have that $0 \leq v(x', x_N) \leq \varphi(x')$ and $\partial_{x_N}v \geq 0$. By monotonicity, we can define:
\[
 \bar{v}(x') = \lim_{x_N \to + \infty} v_n(x', x_N).
 \]

We now use the invariance of the Hamiltonian to reach a contradiction. If we apply Proposition~\ref{H} to $v$ at $t=0$ we obtain:
\[
H = -\frac 1 2 \int_{\omega}( \partial_{x_N}v(x',0))^2 \, dx' \leq 0.
\]
Moreover, being $\bar v\in H^1_0(\omega)$, one has
\begin{align*} 
\lim_{t \to + \infty} H(t) &= \lim_{t \to + \infty}\int_{\omega}	 \left(	\frac 12 \left( |\nabla_{x'} v(x',t)|^2 - (\partial_{x_N}v(x',t))^2 \right) - F(v(x',t)) \right)\, dx' \\ 
	&= \int_{\omega}  \left(	\frac 12 |\nabla_{x'}\bar{v}|^2  - F(\bar{v}(x'))\right) \, dx' = I(\bar{v}) \geq 0. 
	\end{align*}

As a consequence, $ \partial_{x_N}v(x',0)=0$ for all $x' \in \omega$. By unique continuation we would obtain that $v \equiv 0$, a contradiction.

If we assume instead that $n-z_n$ is not diverging we can argue in an analogous manner.

\end{proof}

We are now in conditions to prove Theorem~\ref{tmain}.

\begin{proof}[Proof of Theorem~\ref{tmain}]
	
Define $\hat{u}_n(x', x_N) = u_n(x', x_N - z_n)$, so that $\hat{u}_n(z',0)= \frac 1 2 \| \varphi \|_{L^{\infty}}$. Passing to a subsequence and using Ascoli-Arzelà Theorem, and taking into account Lemma~\ref{lemmino}, we conclude that $\hat{u}_n \to u$ in $ C^{2,\alpha}_{loc}(\Omega)$ sense, where $u$ is a solution of problem~\ref{PB}. By the pointwise convergence of the derivatives we have that $\partial_{x_N} u \geq 0$ This monotonicity allows us to define:
\[
 \underline{u}(x')= \lim_{x_N \to - \infty} u(x', x_N),\quad \overline{u}(x')= \lim_{x_N \to + \infty} u(x', x_N).
 \]

Clearly, $0 \leq \underline{u} <\overline{u} \leq \varphi$ are solutions of the problem:
\begin{equation*}
	\begin{cases}
		-\Delta_{x'}\phi=f(\phi)&\text{in }\omega,\\
		\phi=0&\text{on }\partial\omega.
	\end{cases}
\end{equation*}
Moreover, by pointwise convergence, 
\begin{equation} 
\label{ineq}
\underline{u}(z') \leq {u}(x', 0) =  \frac 1 2  \varphi (z') \leq \overline{u}(z').
\end{equation}

We now make use of the Hamiltonian associated to $u$. Observe that since the Hamiltonian values $H_n$ are negative for all $n \in \N$, passing to the limit we obtain that $H \leq 0$. Then:
\begin{align*}  
0 \geq H &= \lim_{t \to +\infty} \int_{\omega}		 \left(	\frac 12 \left( |\nabla_{x'} u(x',t)|^2 - (\partial_{x_N}u(x',t))^2 \right) - F(u(x',t))	\right) \, dx' \\ 
&=\int_{\omega}	\left(	\frac 12  |\nabla_{x'} \overline{u}(x')|^2   - F(\overline{u}(x'))	\right) \, dx' = I(\overline{u}) \geq 0. 
\end{align*}

Analogously we also obtain that $I(\underline{u}) =0$. By the minimality of $\varphi$ (recall Lemma~\ref{minimal}) and~\ref{ineq} we conclude that $\underline{u}=0$ and $\overline{u}= \varphi$. 

Finally, observe that $\partial_{x_N} u$ is a weak solution of:
\[
- \Delta (\partial_{x_N} u) = f'(u) \partial_{x_N} u.
\]
The maximum principle applies and we obtain that $\partial_{x_N}u >0$.

\end{proof}

\section{On the assumptions~\ref{H1} and~\ref{H2}}
\label{sec3}

As we have shown in the previous section, Theorem~\ref{tmain} holds under the hypotheses~\ref{H1} and~\ref{H2}. In this section our intention is to shed some light on those assumptions. First, we show that we can always find functions $f$ so that these assumptions are satisfied. Second, we show that under~\ref{H1} and~\ref{H2} there exists a positive minimal solution of~\ref{n-1:PB}. Finally, we will see that such assumptions are rather natural: indeed, the existence of a solution as given by Theorem~\ref{tmain} implies that $0$, $\varphi$ are stable solutions of the problem~\ref{n-1:PB} at the same energy level.

\begin{prop}
\label{example}
Given a smooth and bounded domain $\omega\subseteq\R^{N-1}$, we can find a smooth odd function $f\colon\R \to \R$ such that assumptions~\ref{H1} and~\ref{H2} hold true.
\end{prop}

\begin{proof}

The function $f$ claimed in Proposition~\ref{example} will have the form:
\[
f_\lambda(t)= t^3-\lambda t^5,
\]
for some $\lambda >0$ to be found. Clearly its primitive is:
\[
F_\lambda(t)=\int_0^tf_\lambda(s)\,ds= \frac{t^4}{4}-\lambda \frac{t^6}{6}.
\]
Furthermore, we define $I_\lambda\colon H_0^1(\omega) \to (-\infty, +\infty]$,
\[
I_\lambda(\psi)=\int_\omega\left(\frac12\abs{\nabla_{x'}\psi}^2-F_\lambda(\psi)\right)\,dx'= \int_\omega\left(\frac12\abs{\nabla_{x'}\psi}^2  - \frac 1 4 \psi^4 + \frac{\lambda}{6}  \psi^6 \right)\,dx',
\]
and
\[
m_\lambda=\inf_{H^1_0(\omega)}I_\lambda.
\]
Let us point out that $I_\lambda(0)=0$ and then $m_\lambda\le0$, for all $\lambda\ge0$.
Moreover, observe that for any $\lambda$, $0$ is a critical point of $I_\lambda$ and
\begin{equation}
\label{sec:der:Iambda}
I_{\lambda}''(0)[\phi]=\int_\omega\abs{\nabla\phi}^2\,dx',\quad\text{ for all }\phi\in H^1_0(\omega),
\end{equation}
that means it is an isolated, uniformly in $\lambda$, local minimizer.

The proof will be done in several steps.\\

\emph{Step 1: For any $\lambda>0$ the infimum $m_\lambda$ is finite and it is achieved. Moreover, $m_{\lambda} \leq m_{\lambda'}$ if $\lambda \leq \lambda'$.}

For all $\lambda>0$ we have that $F_{\lambda}(t) \geq k_\lambda$ for some constant $k_\lambda \in \R$. Therefore, 
\begin{equation}
\label{coercivity}
I_\lambda(\psi)\ge\frac12\norm{\psi}_{H^1_0(\omega)}^2 - k_{\lambda} \abs{\omega},
\end{equation}
and coercivity easily follows. Then the weak lower semicontinuity is a consequence of the fact that $F_\lambda\ge k_\lambda$ and Fatou's Lemma. Finally, the monotonicity of $m_\lambda$ with respect to $\lambda$ is immediate from its definition.

\medskip 

\emph{Step 2: The set $\mathcal{E}= \set{\lambda>0:m_\lambda<0}$ is not empty and bounded from above; we denote by $\lambda^*$ its supremum.}

We first show that the $\mathcal{E}$ is not empty. Indeed, take any $\phi \in C_0^1(\omega)$; it is clear that for sufficiently large $t >0$, we have that:
\[
\frac{1}{2}\int_\omega\abs{\nabla_{x'}(t \phi)}^2\,dx'- \frac 1 4 \int_\omega (t\phi)^4\,dx'= \frac{t^2}{2}\int_\omega\abs{\nabla_{x'}\phi}^2\,dx'- \frac{t^4}{4} \int_\omega \phi^4\,dx'<-1.
\]

We can take $\lambda>0$ small enough so that $\lambda \frac{t^6}{6} \int_{\omega} \phi^6 <1$ to conclude that $I_{\lambda}(t \phi)<0$. 

Moreover, by the Poincaré inequality we can estimate $I_{\lambda}$ from below as:
\begin{equation}
\label{poincare}
I_\lambda(\psi) \geq \int_\omega \left( \frac{C_{\omega}}{2} \psi^2- \frac 1 4 \psi^4 + \frac{\lambda}{6}  \psi^6 \right)\,dx' .
\end{equation}

Observe that if $\lambda$ is sufficiently large, the function $ t \mapsto  \frac{C_{\omega}}{2} t^2- \frac 1 4 t^4 + \frac{\lambda}{6}  t^6$ is nonnegative, and hence $m_{\lambda}=0$.

\medskip 
\emph{Step 3:~\ref{H1} and~\ref{H2} are satisfied for $\lambda=\lambda^*$.}

We start by pointing out that $m_{\lambda^*}=0$ by the definition of $\lambda^*$.
Consider a sequence $(\lambda_n)_{n\in\N}\subseteq\mathcal{E}$ such that $\lambda_n\to\lambda^*$ as $n\to+\infty$. Then there exists $(\phi_{\lambda_n})_{n\in\N}\subseteq H^1_0(\Omega)$ such that $I_{\lambda_n}(\phi_{\lambda_n})=m_{\lambda_n}<0$ for all $n\in\N$. 
 First of all, let us show that, without loss of generality, we can assume $\phi_{\lambda_n}\ge0$ in $\omega$. Indeed, since $\abs{\phi_{\lambda_n}}\in H^1_0(\omega)$ and $\abs{\nabla\abs{\phi_{\lambda_n}}}=\abs{\nabla\phi_{\lambda_n}}$ a.e. in $\omega$ one has
\begin{align*}
m_{\lambda_n}\le I_{\lambda_n}(\abs{\phi_{\lambda_n}})
	= I_{\lambda_n}(\phi_{\lambda_n})=m_{\lambda_n},
\end{align*}
proving that we can assume $\phi_{\lambda_n}\ge0$ in $\omega$.

Now, taking into account that $\lambda_n\to\lambda^*$, the inequality~\ref{coercivity} tells us that $\norm{\phi_{\lambda_n}}_{H^1_0(\omega)}\le C$ for some constant $C>0$ and for all $n\in\N$. Then, up to subsequences, $\phi_{\lambda_n}\wto\phi^*$ weakly in $H^1_0(\omega)$. We now claim that $\phi^* \neq 0$.
Indeed, let us assume, by contradiction, that $\phi^*\equiv0$. In particular, by Sobolev’s embedding theorem, one has $\phi_{\lambda_n}\to0$ strongly in $L^2(\omega)$. Now, it is trivial to see that for all $\lambda>0$ there exists $C_\lambda>0$ such that $f_\lambda(t)\le C_\lambda$ for all $t\ge0$. In particular $f_{\lambda_n}(t)\le 2C_{\lambda^*}$ for all $t\ge0$ and for all $n\in\N$ large enough. Then, if we let $\bar\phi\in H^1_0(\omega)$ be the solution of
\[
\begin{cases}
-\Delta\bar\phi=2C_{\lambda^*}&\text{in }\omega,\\
\bar\phi=0&\text{on }\partial\omega,
\end{cases}
\]
it is an easy consequence of the maximum principle the fact that $0\le\phi_{\lambda_n}\le\bar\phi$ in $\omega$ for all $n\in\N$ large enough. In particular $0\le\phi_{\lambda_n}\le\bar c$ in $\omega$ for all $n\in\N$ large enough, for some positive constant $\bar c\in\R$. Since $\phi_{\lambda_n}$ are solutions of 
\[
\begin{cases}
-\Delta_{x'}\phi=f_{\lambda_n}(\phi)&\text{in }\omega,\\
\phi=0&\text{on }\partial\omega,
\end{cases}
\]
by classical regularity theory there is a constant $C=C(\omega)>0$ such that
\[
\norm{\phi_{\lambda_n}}_{H^1_0(\omega)}\le C\norm{f(\phi_{\lambda_n})}_{L^2(\omega)}\le C\mathrm{Lip}(f,[0,\bar c])\norm{\phi_{\lambda_n}}_{L^2(\omega)}\to0,
\]
as $n\to+\infty$, but this is in contrast with the fact that $0$ is an isolated, uniformly in $n$, local minimizer of $I_{\lambda_n}$.

Now, since $\phi_{\lambda_n}\to\phi^*$ in $L^2(\omega)$ and $\phi_{\lambda_n}$ is a uniformly bounded sequence, we have that $\phi_{\lambda_n}\to\phi^*$ in $L^p(\omega)$ for any $1<p< +\infty$. In particular, this shows that $I_{\lambda^*} (\phi^*)=0$.

Furthermore, we have that $\phi^*\ge0$ in $\omega$ is a weak solution of the problem: 
\begin{equation}
\label{PBeps*}
\begin{cases}
-\Delta_{x'}\phi=f_{\lambda^*}(\phi)&\text{in }\omega,\\
\phi=0&\text{on }\partial\omega.
\end{cases}
\end{equation}
By classical regularity theory $\phi^*\in C^{2, \alpha}(\omega)$. We can rewrite~\ref{PBeps*} as $-\Delta_{x'}\phi^*+q(x')\phi^*=0$ in $\omega$, where $q$ is defined as:
\[
q= \begin{cases}f_{\lambda^*}(\phi^*)/\phi^* & \mbox{ if }\phi^*\not=0, \\ 0 & \mbox{ elsewhere.} \end{cases}  
\]
 Since $f \in  C^1$ and $f(0)=0$,  $q \in L^{\infty}(\omega)$ and the maximum principle can be applied to show that $\phi^*>0$ in  $\omega$ and the proof of~\ref{H2} is complete.

Finally,~\ref{H1} is an immediate consequence of~\ref{sec:der:Iambda} and the fact that $m_{\lambda^*}=0$.

\end{proof}

\begin{lemma} \label{minimal} Assume that $f \in  C^1(\R)$ is such that~\ref{H1},~\ref{H2} are satisfied. Then, there exists $\varphi \in  C^{2, \alpha}(\omega)$ a positive solution of~\ref{n-1:PB} such that $I(\varphi) =0$, and if $\psi\in H_0^1(\omega)$, $\psi>0$, satisfies $I(\psi)=0$, then $\varphi \leq \psi$. \end{lemma}

\begin{proof}
	
The proof is done in two steps.

\medskip 	

\emph{Step 1: if $\phi,\psi\in H^1_0(\omega)$ satisfy $I(\phi)=I(\psi)=0$ then one of the following is verified
	\[
	\phi<\psi,\quad\text{or}\quad\phi\equiv\psi,\quad\text{or}\quad\psi>\phi,
	\]
	that is the set of minima of $I$ is totally ordered.
}

Let us define $\xi=\min\{\phi,\psi\},\eta=\max\{\phi,\psi\}\in H^1_0(\omega)$. Hence $I(\xi)\ge0$, $I(\eta)\ge0$ and
\[
0\le I(\xi)+I(\eta)=I(\phi)+I(\psi)=0.
\]
We can then infer that $I(\xi)=I(\eta)=0$. In particular, $\xi$ is a solution of problem~\ref{PBeps*} and then by classical regularity theory $\xi\in C^{2,\alpha}(\omega)$. Let us now consider the function $\Xi=\phi-\xi$ that satisfies
\[
\begin{cases}
	-\Delta_{x'}\Xi+q(x')\Xi=0&\text{in }\omega,\\
	\Xi\ge0&\text{in }\omega,
\end{cases}
\]
where
\[
q(x')=
\begin{cases}
	\frac{f(\phi(x'))-f(\xi(x'))}{\phi(x')-\xi(x')},&\text{if }x'\in\{\phi>\psi\},\\
	0,&\text{if }x'\in\{\phi\le\psi\}.
\end{cases}
\]
Since $q\in L^\infty(\omega)$, we can apply the Maximum Principle for nonnegative solutions to infer that $\Xi>0$ in $\omega$ or $\Xi\equiv0$ in $\omega$. If the first case occurs we deduce $\phi>\psi$ in $\omega$, while if the second one is verified it follows $\phi\le\psi$ in $\omega$. 

We can argue in the same way with $\eta$, to conclude that either  $\phi<\psi$ in $\omega$ or $\phi\ge\psi$ in $\omega$. 

\medskip

\emph{Step 2: there exists $\varphi\in C^{2, \alpha}(\omega)$ such that $I(\varphi)=0$, $\varphi>0$ in $\omega$ and if $\psi\in H^1_0(\omega)$, $\psi>0$ satisfies $I_\omega(\psi)=0$ then one of the following is verified
	\[
	\psi\equiv\varphi,\quad\text{or}\quad\psi>\varphi.
	\]}

We can now define $\mathcal M=\set{\psi\in H^1_0(\omega): I(\psi)=0,\quad \psi>0}$, which is not empty thanks to~\ref{H2}, and:
\[
 \alpha = \inf \{ \| \psi\|_{L^{\infty(\omega)}}: \ \psi \in \mathcal{M}\} \geq 0.
 \]

Take $\psi_n \in \mathcal{M}$ with $\| \psi_n \|_{L^{\infty}} \to \alpha$. Since $\psi_n$ is a sequence of uniformly bounded solutions to~\ref{n-1:PB}, Schauder estimates imply that $\psi_n$ is uniformly bounded in $ C^{2, \alpha}(\omega)$. Up to a subsequence we have that $\psi_n \to \varphi$ in  $ C^{2, \alpha}$ sense, and $I(\varphi)=0$. We claim that $\varphi$ is the desired minimal solution.

Since $0$ is an isolated minimizer by~\ref{H1}, we conclude that $\varphi \not \equiv 0$, $\varphi \geq 0$. We argue again as before: we can rewrite~\ref{PBeps*} as $-\Delta_{x'}\varphi+q(x')\varphi=0$ in $\omega$, where $q$ is defined as:
\[
q=\begin{cases} f(\varphi)/\varphi & \mbox{ if }\varphi\not=0, \\ 0 & \mbox{ elsewhere.} \end{cases} 
\]
 Since $f \in  C^1$ and $f(0)=0$,  $q \in L^{\infty}(\omega)$ and the maximum principle can be applied to show that $\varphi>0$ in  $\omega$. In other words, $\varphi \in \mathcal{M}$.

Take now an arbitrary function $\psi\in \mathcal{M}$. By Step 1, we only need to exclude $\psi < \varphi$. But, in such case,
\[
 \|\psi\|_{L^{\infty}(\omega)} < \| \varphi \|_{L^{\infty}(\omega)} = \alpha,
 \]
which is a contradiction with the definition of $\alpha$.

\end{proof}

We now show that, to some extent, the assumptions~\ref{H1} and~\ref{H2} are necessary for Theorem~\ref{tmain} to hold.

\begin{prop}
\label{Hi-nec}
 Assume that $f \in  C^1(\R)$ is so that there exists $u\in C^{2, \alpha}(\overline\Omega)$ solution of~\ref{PB} such that:
\begin{enumerate}
	\item $\displaystyle\lim_{x_N\to-\infty}u(\cdot,x_N)=0$ and $\displaystyle\lim_{x_N\to+\infty}u(\cdot,x_N)=\varphi(\cdot)$,
	uniformly in $x'\in\omega$.
	\item $u$ is strictly increasing in $x_N$, that is 
	\[
\partial_{x_N} u>0\quad\text{in }\Omega.\]
\end{enumerate}

Then, $0$ and $\varphi$ are solutions of~\ref{n-1:PB} with $I(0) = I(\varphi)$. Moreover, $I''(0)$ and $I''(\varphi)$ are semipositive definite forms.

\end{prop}

\begin{proof}
We make use of the invariance of the Hamiltonian (Proposition~\ref{H}) to the function $u$:
\begin{multline*}
H =  \lim_{t \to \pm \infty}\int_{\omega}	\left(	\frac 12 \left( |\nabla_{x'} u(x',t)|^2 - (\partial_{x_N}u(x',t))^2 \right) - F(u(x',t))	\right) \, dx'
 \\ \longrightarrow   \left  \{ \begin{array}{ll} I(0) & \mbox{ if } t \to - \infty, \\ I(\varphi) & \mbox{ if } t \to + \infty. \end{array} \right.
\end{multline*}

Now, observe that $\partial_{ x_N} u$ is a positive solution to the problem:
\[
\begin{cases}
	-\Delta (\partial_{ x_N} u)=f'(u)\partial_{ x_N} u&\text{in }\Omega,\\
	\partial_{ x_N} u=0&\text{on }\partial\Omega.
\end{cases}
\]
As a consequence (see for instance \cite{BNV}) $u$ is stable, meaning that:
\[
Q(\phi) = \int_{\Omega}\left( |\nabla \phi|^2 - f'(u) \phi^2\right)\,dx \geq 0, \quad \text{for all } \phi \in H_0^1(\Omega).
\]

It is well known (see for instance~\cite[Lemma 3.1]{AC00}) that such property is inherited by the upstream and downstream limits $\varphi$ and $0$, respectively. In other words, $I''(0)$ and $I''(\varphi)$ are semipositive definite.

\end{proof}

\section{Applications}
\label{sec4}

As we mentioned in the introduction, the initial motivation of our study was to show the existence of solutions to semilinear elliptic problems without critical points. Indeed, we can prove the following result.

\begin{thm} \label{t.nocrit} There exists $f\colon \R \to \R$ a smooth function and $u$ a bounded smooth solution to the problem:
\begin{equation}
	\label{PB2}
	\begin{cases}
		-\Delta u=f(u)&\text{in }\Omega,\\
		u=0&\text{on }\partial\Omega,
	\end{cases}
\end{equation}
where $\Omega \subseteq \R^2$ is either the strip $(0,1) \times \R$, the half-plane $\R^+ \times \R$, or the whole plane $\Omega = \R^2$ (in this last case the boundary condition is void). 

Moreover $u$ is not a 1D solution (i.e., some level sets are not formed by straight lines), $\nabla u \in L^{\infty}(\Omega)^2$ and $\nabla u(x) \neq 0 $ for all $x \in \overline{\Omega}$.
\end{thm}

\begin{proof}
	For the proof, take $\omega = (0,1)$, and fix $\Omega_0= (0,1) \times \R$. Let $f$ be as given in Proposition~\ref{example}, so that assumptions~\ref{H1} and~\ref{H2} hold. By Theorem~\ref{tmain}, there exists a solution $u_0$ to the problem:
\[
	\label{PB3}
	\begin{cases}
		-\Delta u_0=f(u_0)&\text{in } \Omega_0,\\
		u_0=0&\text{on } \partial \Omega_0,
	\end{cases}
\]
such that $\partial_{x_2}u_0>0$ in $\Omega_0$. Observe moreover that $\partial_{x_1}u_0 \neq 0$ on $\partial \Omega_0$ by the Hopf lemma, since $f$ is smooth and $f(0)=0$. As a consequence $\nabla u_0 $ does not vanish in $\overline{\Omega}_0$ as claimed. The $L^{\infty}$ boundedness of $\nabla u_0$ follows from standard regularity estimates. Furthermore, since $f$ is of class $ C^\infty(\R)$, a bootstrap argument yield that $u\in C^\infty(\overline\Omega)$.

We consider now the case $\Omega_1 = \R^+ \times \R$. The idea is to extend the definition of $u_0$ to $\Omega_1$ by odd reflection along the $x_1$ axis. Being more specific, we define: 
\[
 v: (0,2) \times \R \to \R,  \quad v(x_1, x_2 )=\begin{cases}u_0(x_1, x_2) & \mbox{ if } 0\leq x_1 \leq 1, \\ -u_0(2-x_1, x_2) &  \mbox{ if } 1\leq x_1 \leq 2,\end{cases}
 \]
and
\[
 u_1: \Omega_1 \to \R,  \quad u_1(x_1, x_2) = v(x_1-2 [x_1/2], x_2),
 \]
where $[x]$ denotes the integer part of $x$.

Since $f\colon \R \to \R$ is an odd function, $u_1$ is a solution of~\ref{PB2} with $\Omega= \Omega_1$. Obviously $\nabla u_1 \in L^\infty(\Omega_1)^2$ and $\nabla u_1(x) \neq 0$ for all $x \in \Omega_1$.

For the case $\Omega= \R^2$, it suffices to define:
\[
u_2: \R^2 \to \R,  \quad u_2(x_1, x_2 )=\begin{cases} u_1(x_1, x_2) & \mbox{ if } x_1 \geq 0, \\ -u_1(-x_1, x_2) &  \mbox{ if } x_1 \leq 0.\end{cases}
\]
Again, $-\Delta u_2 = f(u_2)$ by the oddness of $f$, $\nabla u_2 \in L^{\infty}(\R^2)^2$ and $\nabla u_2(x) \neq 0 $ for all $x \in \R^2$.
 
\end{proof}

\begin{rmks}
\label{rmk:dg} 
$ $ 
\begin{enumerate}[$\bullet$]
\item Obviously we can obtain solutions which are not 1D and without critical points in higher dimensions, that is, in domains $\Omega= (0,1) \times \R^{N-1}$, in half-spaces $\R^+\times\R^{N-1}$ and in the whole space $\R^N$, by extending the solutions constantly with respect to the remaining variables.

\item Let us notice that the solutions given in Theorem~\ref{t.nocrit} satisfy:
\[
\inf_\Omega\abs{\nabla u}=0.
\]
Indeed, recall that $\lim_{x_N \to -\infty} u(x', x_N) =0$ in $ C^{2, \alpha}$ sense.

\item In the case of entire solutions, since $\nabla u (x) \neq 0$ for all $x \in \R^2$, we can write $\nabla u (x)= \rho(x) e^{i \theta(x)}$, by using complex notation. In such case we have that (\cite{Far03}):
\[
\divv (\rho^2 \nabla \theta)=0.
\]
Observe that $\theta(0,x_2)=0$, $\theta(1,x_2)= \pi$ for all $x_2 \in \R$; by reflection we obtain that $\theta$ grows linearly in $|x|$, indeed,
\[
 \limsup_{|x| \to + \infty }  \frac{|\theta (x)|}{|x|} = \pi.
 \]
This is to be compared with the results in~\cite{Salva}.
\end{enumerate}

\end{rmks}

The previous result implies, in particular, the existence of bounded solutions to the 2D Euler equations without stagnation points which are not shear flows, as shown in the next corollary.

\begin{cor} 
\label{cor}
Let us consider the stationary 2D Euler equations:
\begin{equation}
	\label{Eu2}
	\begin{cases}
		\eu\cdot\nabla \eu=-\nabla p&\text{in }\Omega,\\
		\divv \eu=0&\text{in }\Omega, \\
		\eu \cdot \eta =0 & \text{on } \partial \Omega,
	\end{cases}
\end{equation}	
where $\Omega \subseteq \R^2$ is either the strip $(0,1) \times \R$, the half-plane $\R^+ \times \R$, or the whole plane $\Omega = \R^2$ (in this last case the boundary condition is void). There exist smooth solutions to those problems such that $\eu \in L^{\infty}(\Omega)^2$ and $\eu (x) \neq 0$ for all $x \in \overline{\Omega}$, which are not shear flows.
		 
\end{cor}

\begin{proof}
Let us consider the solution $u$ given in Theorem~\ref{t.nocrit}. It is well known that the vector field $\eu=\nabla^\perp u=(-\partial_{x_2}u,\partial_{x_1}u)$ satisfies~\ref{Eu2} in $\Omega$, with pressure given by
\[
p=-\frac{\abs{\nabla{u}}^2}{2}-F(u).
\]

The function $u$ is called the streamfunction of the flow $\eu$. Since $u$ is not 1-dimensional, $\eu$ is not a shear flow. 

\end{proof}

%

\bibliographystyle{abbrv}
\bibliography{Euler.bib}

\bigskip 

{\bf Disclosure statement: }  The authors report there are no competing interests to declare. 
\bigskip

\end{document}